\documentclass[reqno]{amsart}

\usepackage{amsmath,amssymb,amsthm,enumerate,color}

\newcommand{\lsg}{\textnormal{LSG}}
\newcommand{\V}{\textnormal{V}}

\newcommand{\Cay}{\textnormal{Cay}}

\newcommand{\diam}{\textnormal{diam}}
\newcommand{\Sym}{\textnormal{Sym}}
\newcommand{\Alt}{\textnormal{Alt}}
\newcommand{\Aut}{\textnormal{Aut}}

\newcommand{\Fix}{\textnormal{Fix}}
\newcommand{\GL}{\textnormal{GL}}
\newcommand{\SL}{\textnormal{SL}}
\newcommand{\GamL}{\textnormal{$\Gamma$L}}
\newcommand{\AGL}{\textnormal{AGL}}
\newcommand{\Sp}{\textnormal{Sp}}
\newcommand{\Tr}{\textnormal{Tr}}
\newcommand{\F}{\mathbb{F}}

\newcommand{\mymod}[1]{\,\left(\textnormal{mod}\ {#1}\right)}
\newcommand{\bigmod}[1]{\,\big(\textnormal{mod}\ {#1}\big)}

\newtheorem{theorem}{Theorem}[section]
\newtheorem{lemma}[theorem]{Lemma}

\theoremstyle{definition}
\newtheorem*{notation}{Notation}

\newtheorem{example}[theorem]{Example}
\newtheorem{hypothesis}[theorem]{Hypothesis}

\title{Quotient-complete arc-transitive latin square graphs from groups}
\author{Carmen Amarra}
\address{Institute of Mathematics, University of the Philippines Diliman}
\address{Natural Sciences Research Institute, College of Science, University of the Philippines Diliman}
\date{\today}

\begin{document}

\begin{abstract}
We consider latin square graphs $\Gamma = \lsg(H)$ of the Cayley table of a given finite group $H$. We characterize all pairs $(\Gamma,G)$, where $G$ is a subgroup of autoparatopisms of the Cayley table of $H$ such that $G$ acts arc-transitively on $\Gamma$ and all nontrivial $G$-normal quotient graphs of $\Gamma$ are complete. We show that $H$ must be elementary abelian and determine the number $k$ of complete normal quotients. This yields new infinite families of diameter two arc-transitive graphs with $k = 1$ or $k = 2$.
\end{abstract}

\maketitle

\section{Introduction}

A graph $\Gamma$ with automorphism group $G \leq \Aut(\Gamma)$ is said to be \emph{quotient-complete} if each proper $G$-normal quotient of $\Gamma$ is either complete or empty, and $\Gamma$ has at least one nontrivial complete $G$-normal quotient. In this paper we classify all pairs $(\Gamma,G)$, where $\Gamma$ is a latin square graph of a finite group and $G$ is a subgroup of autoparatopisms of $\Gamma$, such that $\Gamma$ is $G$-arc-transitive and $G$-quotient-complete.

This paper is part of a general study of arc-transitive, diameter $2$ graphs carried out in \cite{AGP-quotcomp, HAcase, HApaper}. The family $\mathcal{F}$ of all such graphs is analyzed using normal quotient reduction, and it is shown in \cite[Theorem 2.2]{AGP-quotcomp} that any graph in $\mathcal{F}$ has a normal quotient graph $\Gamma$ with automorphism group $G$ such that $\Gamma$ is either \emph{$G$-vertex-quasiprimitive} (i.e., all nontrivial normal subgroups of $G$ are transitive on vertices) or $G$-quotient-complete. A subclass of vertex-quasiprimitive graphs were studied in \cite{HAcase} and \cite{HApaper}. Quotient-complete graphs were studied in \cite{AGP-quotcomp}, in which it was shown that a significant parameter of quotient-complete graphs is the number $k$ of distinct nontrivial, complete normal quotients. The graphs for which $k \geq 3$ were classified in \cite{AGP-quotcomp} (except for graphs corresponding to subgroups of the one-dimensional affine group). For $k = 1$ and $k = 2$ some infinite families are known, but a classification was not achieved. These infinite families involve products of graphs, and are described in Examples \ref{example:lexicographic} and \ref{example:directproduct}. In addition, another family of arc-transitive, quotient-complete graphs with $k = 2$ was communicated to the author by P. Spiga. These graphs are strongly regular and are described in Example \ref{example:Cay}.

Our goal is to further examine the cases where $k = 1$ and $k = 2$, and in particular to find other infinite families of examples satisfying these. It is known that the graphs in Example \ref{example:Cay} are graphs of Bruck nets \cite{Bruck}, all of which are strongly regular and thus have diameter $2$. Hence a natural problem is to determine which Bruck nets give rise to quotient-complete graphs. The graph in Example \ref{example:Cay} corresponding to the parameter $q = 7$ is a latin square graph of the Cayley table of the cyclic group $C_7$; it is known that all latin squares are Bruck nets. It was shown in \cite{Bruck} that for latin squares of size at least $5$, any automorphism of the latin square graph is induced by an autoparatopism of the underlying latin square. In particular, the Cayley tables of finite groups form a family of latin square graphs whose autoparatopism groups are known \cite{Heinze}.

Our main results are Theorems \ref{theorem:mainthm1} and \ref{theorem:mainthm2}.

\begin{theorem} \label{theorem:mainthm1}
Let $H$ be a finite group and $\Gamma = \lsg(H)$. The graph $\Gamma$ is $G$-arc-transitive and $G$-quotient-complete, for some subgroup $G$ of autoparatopisms of the Cayley table of $H$, only if $H$ is elementary abelian, say $H = C^d_p$ for some prime $p$ and positive integer $d$, satisfying one of the following: $p \not\equiv 2 \mymod{3}$, $d$ is even, or $d \equiv 0 \mymod{3}$. Otherwise, $\Gamma$ is $G$-vertex-quasiprimitive. Conversely, for any $H = C^d_p$ satisfying the above conditions, there exists a subgroup $G$ of autoparatopisms such that $\Gamma$ is $G$-arc-transitive and $G$-quotient-complete.
\end{theorem}

\begin{theorem} \label{theorem:mainthm2}
Let $\Gamma = \lsg(H)$ and $G \leq \Aut(\Gamma)$, where $H = C^d_p$ for some prime $p$ and positive integer $d$, and $G$ is a subgroup of autoparatopisms of the Cayley table of $H$. If $\Gamma$ is $G$-arc-transitive and $G$-quotient-complete then any nontrivial complete $G$-normal quotient of $\Gamma$ has $p^d$ vertices. Furthermore, $\Gamma$ has exactly $k$ nontrivial complete $G$-normal quotients, where $k = 1$ or $k \geq 3$ if $p = 3$, and $k \geq 2$ if $p \neq 3$.
\end{theorem}

The possible pairs $(\Gamma,G)$ such that $\Gamma$ is $G$-arc-transitive and $G$-quotient-complete are described in Theorem \ref{theorem:G_0} and Theorem \ref{theorem:maintheorem}.

The rest of this paper is organized as follows. In Section \ref{section:prelims} we present some background on quotient-complete graphs and latin square graphs, as well as some technical results on divisibility and finite fields that will be useful in the next section. In Section \ref{section:arctranslsg} we establish general facts about arc-transitive latin square graphs from finite groups; in particular we show that the underlying group must be elementary abelian. In Section \ref{section:quotcomplsg} we determine which elementary abelian groups and their corresponding automorphism groups yield quotient-complete arc-transitive graphs, and prove Theorems \ref{theorem:maintheorem}, \ref{theorem:mainthm1}, and \ref{theorem:mainthm2}.
\section{Preliminaries} \label{section:prelims}

\begin{notation}
If $S$ is a group, vector space, or finite field, $S^\#$ denotes the set of non-identity or nonzero elements of $S$. If $n$ is a positive integer, $\Sym(\Omega)$ and $\Sym(n)$ denote, respectively, the symmetric group on a set $\Omega$ and the symmetric group on $n$ letters.
\end{notation}

\subsection{Quotient-complete graphs}

Given a graph $\Gamma$, a group $G \leq \Aut(\Gamma)$, and $N \unlhd G$, the \emph{$G$-normal quotient} $\Gamma_N$ of $\Gamma$ is the graph whose vertices are the $N$-orbits in $\V(\Gamma)$ and with adjacency defined as follows: two $N$-orbits $O_1$ and $O_2$ are adjacent in $\Gamma_N$ exactly when there are vertices $v_1 \in O_1$ and $v_2 \in O_2$  such that $v_1$ and $v_2$ are adjacent in $\Gamma$. The quotient group $G/N$ is a subgroup of $\Aut(\Gamma_N)$. A normal quotient graph inherits some of the properties of the original graph, such as connectedness, vertex-transitivity, and arc-transitivity. In the case where $\Gamma$ is connected, the normal quotient $\Gamma_N$ has diameter not exceeding that of $\Gamma$. Thus, if $\diam(\Gamma) = 2$, then either $\Gamma_N$ is a complete graph or $\diam(\Gamma_N) = 2$. 
It is proved in \cite[Theorem 2.2]{AGP-quotcomp} that quotient-complete graphs arise naturally as normal quotients of vertex-transitive graphs.

Let $k$ be the number of distinct nontrivial complete $G$-normal quotients of $\Gamma$. It was shown in \cite{AGP-quotcomp} that if $\Gamma$ is $G$-arc-transitive and $G$-quotient-complete with $k \geq 3$, then $|\V(\Gamma)| = c^2$ for some prime power $c$ and $|\V(\Gamma_N)| = c$ for any nontrivial $G$-normal-quotient $\Gamma_N$. Furthermore either $\Gamma$ is isomorphic to $c$ copies of the complete graph $K_c$ on $c$ vertices (in which case $k = c$), or $k = c' + 1$ for some divisor $c'$ of $c$. For the cases $k = 1$ and $k = 2$ infinite families of examples are obtained via the following two constructions.

\begin{example} \cite[Example 3.1]{AGP-quotcomp} \label{example:lexicographic}
Let $\Gamma = K_m \left[ \overline{K_n} \right]$ and $G = \Sym(n) \wr \Sym(m)$ for positive integers $m$ and $n$. The graph $\Gamma$ is the \emph{lexicographic product} of $K_m$ and the empty graph $\overline{K_n}$ on $n$ vertices; the vertex set is $\{1, \ldots, m\} \times \{1, \ldots, n\}$ and the edges are the pairs $\{(i,j), (i',j')\}$ where $i \neq i'$. Then $\Gamma$ is a connected graph with $G = \Aut(\Gamma)$. The graph $\Gamma$ is $G$-arc-transitive and $G$-quotient-complete with $k = 1$, and the unique $G$-normal quotient is $K_m$, which corresponds to $N = \Sym(n)^m$.
\end{example}

\begin{example} \cite[Example 3.2]{AGP-quotcomp} \label{example:directproduct}
Let $\Gamma = K_m \times K_n$ and $G = \Sym(m) \times \Sym(n)$ for positive integers $m$ and $n$. The graph $\Gamma$ is the \emph{direct product} of the complete graphs $K_m$ and $K_n$; the vertex set is $\{1, \ldots, m\} \times \{1, \ldots, n\}$ and the edges are the pairs $\{(i,j), (i',j')\}$ where $i \neq i'$ and $j \neq j'$. Then $\Gamma$ is a connected graph with $G \leq \Aut(\Gamma)$. The graph $\Gamma$ is $G$-arc-transitive and $G$-quotient-complete with $k = 2$. The two $G$-normal quotients are $K_m$ and $K_n$, which correspond to $M = \Sym(n)$ and $N = \Sym(m)$, respectively.
\end{example}

Observe that in Example \ref{example:directproduct} the quotient graphs $\Gamma_M$ and $\Gamma_N$ could have any number of vertices, possibly distinct; in this respect the case where $k = 2$ differs from the case where $k \geq 3$, since in the latter all nontrivial normal quotients have the same order.

In addition to Example \ref{example:directproduct}, another infinite family of quotient-complete graphs with $k = 2$ was introduced to the author by P. Spiga; we describe these in Example \ref{example:Cay}.

\begin{example} \label{example:Cay}
Let $\Gamma = \Cay(V,S)$ and $G = V \rtimes G_0 \leq \AGL(V)$, where $V = \F_q \oplus \F_q$ (with elements written as ordered pairs), $q \geq 5$ is a power of an odd prime, $\F_q$ is the finite field of order $q$, 
	\[ G_0 = \left\{ \left( \begin{array}{cc} a & 0 \\ 0 & b \end{array} \right) \ \vline \ ab \in \F^\square_q \right\}, \]
where $\F^\square_q = \big\{ c^2 \ \big| \ 0 \neq c \in \F_q \big\}$, and $S = (1,1)^{G_0}$. For $q \geq 5$ the graph $\Gamma$ is connected with $G \leq \Aut(\Gamma)$. The graph $\Gamma$ is connected, $G$-arc-transitive, and $G$-quotient-complete with $k = 2$ if $q \geq 5$. Each $G$-normal quotient is isomorphic to $K_q$ and corresponds to the subgroup of translations by elements of $\F_q \oplus \{0\}$ and of $\{0\} \oplus \F_q$. 
\end{example}

If $q = 7$ then $\Gamma \cong \Gamma' := \Cay(V,S')$ where $S' = \big\{ (a,0), (0,a), (a,-a) \ \big| \ a \in \F^\#_7 \big\}$. Indeed, $f := \left( \begin{array}{cc} -1 & 3 \\ 2 & 4 \end{array} \right) \in \GL_2(7)$ induces an automorphism from $\Gamma$ to $\Gamma'$; the group
	\[ f^{-1} G_0 f = \left\langle Z(\GL_2(7)), \ \left(\begin{array}{cc} 1 & -1 \\ 1 & 0 \end{array}\right) \right\rangle \]
stabilizes the vertex $(0,0)$ in $\Gamma'$. We can generalize this to obtain another infinite family of examples with $k \leq 2$, which we describe in Example \ref{example:lsg}. It is easy to see, by comparing valencies, that in general these graphs are distinct from those described above.

\begin{example} \label{example:lsg}
Let $\Gamma = \Cay(V,S)$ and $G = V \rtimes G_0 \leq \AGL(V)$, where $V = \F_q \oplus \F_q$, $q$ a prime power,
	\[ G_0 = \left\langle Z(\GL_2(q)), \ \left(\begin{array}{cc} 1 & -1 \\ 1 & 0 \end{array}\right) \right\rangle, \]
and $S = (1,0)^{G_0} = \big\{ (a,0), (0,a), (a,-a) \ \big| \ a \in \F^\#_q \big\}$. Then $\Gamma$ is connected $G$-arc-transitive, and $G$ has an intransitive minimal normal subgroup exactly when $q \not\equiv 2 \mymod{3}$. In particular, the intransitive minimal normal subgroups of $G$ are subgroups of translations $T_W$ by elements of a subspace $W \leq V$, where $W = \big\{ (a,ca) \ \big| \ a \in \F^\#_q \big\}$ with $c^2 + c + 1 = 0$. Hence $\Gamma$ is $G$-quotient-complete with $k = 1$ if $q \equiv 0 \mymod{3}$ and $k = 2$ if $q \equiv 1 \mymod{3}$. All quotient graphs corresponding to these subgroups $T_W$ are complete graphs $K_q$.
\end{example}

All graphs in Example \ref{example:lsg} are \emph{latin square graphs}, which we discuss in the next subsection.

\subsection{Latin square graphs from groups} \label{subsection:lsg}

Let $n$ be a positive integer. A \emph{latin square} $L$ of order $n$ is an $n \times n$ array of $n$ symbols, such that no symbol occurs twice in the same row or in the same column. If $H$ denotes the set of $n$ symbols, then, by labelling the rows and columns by elements of $H$, each cell in $L$ can be represented by an ordered triple $(h_1, h_2, h_3)$ of elements of $H$, where $h_1$ denotes the row label, $h_2$ the column label, and $h_3$ the symbol contained in the cell. An \emph{autoparatopism} of $L$ is an element of $\Sym(H) \wr \Sym(3)$, in its product action on $H^3$, which preserves $L$ setwise; an \emph{autotopism} of $L$ is an autoparatopism which belongs in $\Sym(H)^3$. Thus an autotopism is an ordered triple of permutations acting on the set of row labels, the set of column labels, and the set of symbols, while an autoparatopism consists of an autotopism followed by a permutation of the three coordinates. In particular, we denote any autoparatopism of $L$ by $[(\sigma_1, \sigma_2, \sigma_3), \gamma]$, where $(\sigma_1, \sigma_2, \sigma_3) \in \Sym(H)^3$ and $\gamma \in \Sym(3)$, with action given by
	\begin{equation} \label{eq:action}
	(h_1, h_2, h_3)^{[(\sigma_1, \sigma_2, \sigma_3), \gamma]} = \left(h_{1'}^{\sigma_{1'}}, h_{2'}^{\sigma_{2'}}, h_{3'}^{\sigma_{3'}}\right), \quad i' := i^{\gamma^{-1}} \; \forall\; i \in \{1,2,3\}
	\end{equation}
for any $(h_1, h_2, h_3) \in L$. The set of all autoparatopisms of $L$ forms a subgroup of $\Sym(H) \wr \Sym(3)$, having as normal subgroup the group of all autotopisms.

A \emph{latin square graph} is a graph $\Gamma$ whose vertices are the cells of a latin square $L$, and whose edges are those pairs of cells which lie in the same row, in the same column, or contain the same symbol. In other words, if $L$ is viewed as a subset of $H^3$, then the edges of $\Gamma$ are those pairs of triples that agree in exactly one coordinate. If $H$ is the set of symbols, we denote $\Gamma$ by $\lsg(H)$. Each autoparatopism of $L$ induces an automorphism of its latin square graph. It was shown in \cite{Bruck} that $\Aut(\Gamma)$ coincides with the group of autoparatopisms of $L$ whenever $|H| \geq 5$. 

From now on we assume that $H$ is a group and that $\lsg(H)$ is the latin square graph of the Cayley table of $H$. Then the vertices of $\lsg(H)$ are triples $(h_1,h_2,h_1h_2)$ for all $h_1, h_2 \in H$, and hence can be associated with $H^2$. The neighbors of any vertex $(a, b, ab)$ are the vertices $(a, h, ah)$ for all $b \neq h \in H$, $(h, b, hb)$ for all $a \neq h \in H$, and $\left(ah, h^{-1}b, ab\right)$ for all $1_H \neq h \in H$. In addition, it can be shown that $\lsg(H)$ is isomorphic to the Cayley graph $\Cay(H^2,S)$ on $H^2$ with
	\[ S = \big\{ (1_H, h), (h, 1_H), (h, h) \ | \ 1_H \neq h \in H \big\}, \]
which is the graph with vertex set $H^2$ and edges $\{s,t\}$ where $st^{-1} \in S$, via the map which sends any $(a, b, ab) \in \V(\Gamma)$ to $\left(a, b^{-1}\right) \in H^2$. The autotopism group of $\lsg(H)$ is generated by the triples $(\lambda_a, \rho_b, \lambda_a\rho_b)$ for any $a, b \in H$, where
	\begin{equation} \label{equation:regularaction}
	\lambda_a : h \mapsto a^{-1}h \quad \text{and} \quad \rho_b : h \mapsto hb \quad \text{for all $h \in H$},
	\end{equation}
and $(\sigma, \sigma, \sigma)$ for any $\sigma \in \Aut(H)$. The elements $(\lambda_a, \rho_b, \lambda_a\rho_b)$ generate a group $T$ isomorphic to $H^2$, and the elements $(\sigma, \sigma, \sigma)$ generate a group $\overline{\Aut(H)}$ that normalizes $T$. Hence the autotopism group of $\lsg(H)$ is isomorphic to $H^2 \rtimes \Aut(H)$, where $\Aut(H)$ acts on $H^2$ componentwise (\cite[Proposition 2]{Bailey}). The elements $x$ and $y$ given in Table \ref{table:autoparatopisms}, where $\iota$ is the identity in $\Sym(H)$ and $\phi \in \Sym(H)$ is defined by
	\begin{equation} \label{equation:maptoinverse}
	\phi : h \mapsto h^{-1} \quad \text{for all $h \in H$},
	\end{equation}
are autoparatopisms of $\lsg(H)$ which generate a subgroup isomorphic to $\Sym(3)$. The autoparatopism group of $\lsg(H)$ is the group $\mathcal{G}$ generated by $T$, $\overline{\Aut(H)}$, and the elements $x$ and $y$. Thus $\mathcal{G} \cong \left(H^2 \rtimes \Aut(H)\right) \Sym(3) \leq \Aut(\lsg(H))$, and it follows from the above that $\mathcal{G} = \Aut(\lsg(H))$ if and only if $|H| \geq 5$.

\begin{table}
\begin{center}
\begin{tabular}{rcll}
\hline
\multicolumn{3}{c}{Autoparatopism} & Image of $(a, b, ab)$ \\
\hline\hline
$x$ & : & $[ (\iota, \phi, \phi), (1\,2\,3) ]$ & $\left(b^{-1}a^{-1}, a, b^{-1}\right)$ \\
$x^2$ & : & $[ (\phi, \iota, \phi), (1\,3\,2) ]$ & $\left(b, b^{-1}a^{-1}, a^{-1}\right)$ \\
$y$ & : & $[ (\iota, \phi, \iota), (1\,3) ]$ & $\left(ab, b^{-1}, a\right)$ \\
$xy$ & : & $[ (\phi, \phi, \phi), (1\,2) ]$ & $\left(b^{-1}, a^{-1}, b^{-1}a^{-1}\right)$ \\
$x^2y$ & : & $[ (\phi, \iota, \iota), (2\,3) ]$ & $\left(a^{-1}, ab, b\right)$ \\
\hline
\end{tabular}
\caption{Some autoparatopisms of the latin square of $H$} \label{table:autoparatopisms}
\end{center}
\end{table}
\subsection{Some technical results on finite fields}

In this subsection we present some technical results on roots of certain polynomials in finite fields, which will be used in Section \ref{section:quotcomplsg}.

We begin with some elementary divisibility results. For nonzero integers $r$ and $s$, the symbol $(r,s)$ denotes the greatest common divisor of $r$ and $s$.

\begin{lemma} \label{lemma:divisibility}
Let $a$, $r$, and $s$ be positive integers with $a \geq 2$. Then
	\begin{enumerate}[1.]
	\item \label{div2} $\big(a^r - 1, \ a^s + 1\big) = 
		\left\{\begin{aligned}
		&a^{(r,s)} + 1 &&\text{if $r/(r,s)$ is even} \\
		&2 &&\text{if $r/(r,s)$ and $a$ are odd} \\
		&1 &&\text{if $r/(r,s)$ is odd and $a$ is even}
		\end{aligned}\right.$
	\item \label{div3} $\big(a^r - 1, \ a^{2s} + a^s + 1\big) =
		\left\{\begin{aligned}
		&a^{2(r,s)} + a^{(r,s)} + 1 &&\text{if $3 \mid r/(r,s)$ and $s/(r,s)$ is even}; \\
		&3 &&\text{if $3 \nmid r/(r,s)$, and either $a \equiv 1 \mymod{3}$, or} \\[-3pt]
		&  &&\text{$a \equiv 2 \mymod{3}$ and $r$ and $s$ are even}; \\
		&1 &&\text{otherwise.}
		\end{aligned}\right.$
	\item \label{div4} $\big(a^r + 1, \ a^{2s} + a^s + 1\big) =
		\left\{\begin{aligned}
		&a^{2(r,s)} - a^{(r,s)} + 1 &&\text{if $3 \mid r/(r,s)$ and $s/(r,s)$ is even}; \\
		&3 &&\text{if $3 \nmid r/(r,s)$, $s/(r,s)$ is even, $r$ is odd,} \\[-3pt]
		&  &&\text{and $a \equiv 2 \mymod{3}$}; \\
		&1 &&\text{otherwise.}
		\end{aligned}\right.$
	\end{enumerate}
\end{lemma}

\begin{proof}
For part \ref{div2} let $D = \big(a^r - 1, \ a^s + 1\big)$. Then $D \,\,\big|\, \big(a^r - 1, \ a^{2s} - 1\big)$, where $\big(a^r - 1, \ a^{2s} - 1\big) = a^{(r,2s)} - 1$ by an easy exercise in elementary number theory (see, for instance, \cite[Exercise 2.2 (6.)]{Rosen}). If $r/(r,s)$ is odd then $(r,2s) = (r,s)$, so that $D \,\,\big|\, a^{(r,s)} - 1$. Hence $a^s + 1 \equiv 2 \mymod{D}$ and thus $D \mid 2$. In particular, $D = 1$ if $a$ is even and $D = 2$ if $a$ is odd. If $r/(r,s)$ is even then $s/(r,s)$ is odd, and it follows that $a^{(r,s)} + 1 \,\,\big|\, D$. Also $(r,2s) = 2(r,s)$ so that $a^{2(r,s)} \equiv 1 \mymod{d}$. Since $a^s \equiv -1 \mymod{D}$ this implies that $a^{(r,s)} \equiv -1 \mymod{D}$. Therefore $D \,\,\big|\, a^{(r,s)} + 1$, and hence $D = a^{(r,s)} + 1$.

For the remainder we will use the easily verified facts that $a^2 + a + 1 \,\,\big|\, a^{2r} + a^r + 1$ exactly when $3 \nmid r$, and $a^2 - a + 1 \,\,\big|\, a^{2r} + a^r + 1$ exactly when $r$ is even and $3 \nmid r$.

For part \ref{div3} let $D = \big(a^r - 1, \ a^{2s} + a^s + 1\big)$. Then $D \,\,\big|\, \big(a^r - 1, \ a^{3s} - 1\big) = a^{(r,3s)} - 1$. If $3 \mid r/(r,s)$ then $3 \nmid s/(r,s)$, so that $a^{2(r,s)} + a^{(r,s)} + 1 \,\,\big|\, D$. Also $(r,3s) = 3(r,s)$, so that $a^{3(r,s)} \equiv 1 \mymod{D}$ and $a^s \equiv a^{t(r,s)} \mymod{D}$, $t \in \{1,2\}$. It follows that $a^{2s} + a^s + 1 \equiv a^{2(r,s)} + a^{(r,s)} + 1 \mymod{D}$, and hence $D \,\,\big|\, a^{2(r,s)} + a^{(r,s)} + 1$. Therefore $D = a^{2(r,s)} + a^{(r,s)} + 1$. If $3 \nmid r/(r,s)$ then $(r,3s) = (r,s)$ so that $a^{(r,s)} \equiv 1 \mymod{D}$ and $a^{2s} + a^s + 1 \equiv 3 \mymod{D}$. Hence $D \mid 3$. In particular, $D = 3$ if and only either if $a \equiv 1 \mymod{3}$, or $a \equiv 2 \mymod{3}$ and both $r$ and $s$ are even.

For part \ref{div4} let $D = \big(a^r + 1, \ a^{2s} + a^s + 1\big)$. Then $D \,\,\big|\, \big(a^r + 1, \ a^{3s} - 1\big)$. If $s/(r,s)$ is odd then so is $3s/(r,3s)$, and it follows from part \ref{div2} that $D \mid 2$. But $a^{2s} + a^s + 1$ is odd, hence $D = 1$. Suppose that $s/(r,s)$ is even. Then so is $3s/(r,3s)$, and by part \ref{div2} we have $\big(a^r + 1, \ a^{3s} - 1\big) = a^{(r,3s)} + 1$. Also $r/(r,s)$ is odd. If $3 \mid r/(r,s)$ then $3 \nmid s/(r,s)$, so $a^{2(r,s)} - a^{(r,s)} + 1 \,\,\big|\, D$. Furthermore $(r,3s) = 3(r,s)$, so that $a^{3(r,s)} \equiv -1 \mymod{D}$ and $a^s \equiv a^{t(r,s)} \mymod{D}$, $t \in \{2,4\}$. Hence $a^{2s} + a^s + 1 \equiv a^{2(r,s)} - a^{(r,s)} + 1$, and so $D \,\,\big|\, a^{2(r,s)} - a^{(r,s)} + 1$. Thus $D = a^{2(r,s)} - a^{(r,s)} + 1$. If $3 \nmid r/(r,s)$ then $(r,3s) = (r,s)$, so that $a^{2s} + a^s + 1 \equiv 3 \mymod{D}$. Hence $D \mid 3$, and since $s$ is even, we have $D = 3$ if and only if $a \equiv 2 \mymod{3}$ and $r$ is odd.
\end{proof}

Let $\omega$ be a primitive element in the finite field $\F_q$ with order $q = p^d$, $p$ prime, and let $\tau_p$ be the Frobenius automorphism on $\F_q$. Then $\left|\tau^i_p\right| = d/(d,i)$ and $|\Fix(\langle \tau^i_p \rangle)| = p^{(d,i)}$. Define the sets
	\begin{equation} \label{equation:C1(p,d,i)}
	C_1(p,d,i) := \left\{ c \in \F^\#_q \ \vline \ c^{p^i} = c^{\tau^i_p} = -c^{-1}(c+1) \right\}.
	\end{equation}
and
	\begin{equation} \label{equation:C2(p,d,i)}
	C_2(p,d,i) := \left\{ c \in \F^\#_q \ \vline \ c^{p^i} = c^{\tau^i_p} = -c(c+1)^{-1} \right\}.
	\end{equation}
Note that any $c \in C_1(p,d,i)$ satisfies $c^{p^{2i}} = c^{\tau^{2i}_p} = -(c+1)^{-1}$ and $c^{\tau^{3i}_p} = c$, and any $c \in C_2(p,d,i)$ satisfies $c^{\tau^{2i}_p} = c$.

For any subfields $\mathbb{E}$ and $\mathbb{K}$ of $\F_q$ with $\mathbb{K} \leq \mathbb{E}$, let $\Tr_{\mathbb{E}/\mathbb{K}}$ denote the trace map from $\mathbb{E}$ to $\mathbb{K}$.

\begin{lemma} \label{lemma:C1-C2}
Let $p$ be a prime, $d$ and $i$ positive integers, and $q = p^d$. Let $\tau_p$ be the Frobenius map on $\F_q$ and let $C_1(p,d,i)$ and $C_2(p,d,i)$ be as in (\ref{equation:C1(p,d,i)}) and (\ref{equation:C2(p,d,i)}), respectively.
	\begin{enumerate}[1.]
	\item If $\left|\tau^i_p\right| \not\equiv 0 \mymod{3}$, then $C_1(p,d,i) = \big\{ c \in \F^\#_q \ \big| \ |c| = 3 \big\}$ if $p^{(d,i)} \equiv 1 \pmod{3}$, $C_1(p,d,i) = \{1\}$ if $p = 3$, and $C_1(p,d,i) = \varnothing$ otherwise.
	\item If $\left|\tau^i_p\right| \equiv 0 \mymod{3}$ then $C_1(p,d,i) = \big\{ b^{p^i-1} \ \big| \ b \neq 0, \ b + b^{p^i} + b^{p^{2i}} = 0 \big\}$ and $|C_1(p,d,i)| = p^{(d,i)} + 1$.
	\item If $\left|\tau^i_p\right|$ is odd then $C_2(p,d,i) = \varnothing$ if $p = 2$ and $C_2(p,d,i) = \{-2\}$ if $p \geq 3$.
	\item If $\left|\tau^i_p\right|$ is even then $C_2(p,d,i) = \big\{ b-1 \ \big| \ b \neq 1, \ b^{p^i + 1} = 1 \big\}$ and $|C_2(p,d,i)| = p^{(d,i)}$.
	\end{enumerate}
\end{lemma}

\begin{proof}
Assume that $\left|\tau^i_p\right| \not\equiv 0 \mymod{3}$. Then $c \in C_1(p,d,i)$ if and only if either $c^{\tau^i_p} = c$ or $c^{\tau^{2i}_p} = c$; both are equivalent to $c^{p^i} = c$ and $c^2 + c + 1 = 0$. Statement 1 follows.

Assume now that $3 \,\big|\, |\tau^i_p|$. It is easy to verify that $\big\{ b^{p^i-1} \ \big| \ b \in \F^\#_q, \ b + b^{p^i} + b^{p^{2i}} = 0 \big\} \subseteq C_1(p,d,i)$. Let $c = \omega^r \in C_1(p,d,i)$. Then
	\[ c^{1 + p^i + p^{2i}} = c \cdot \big(-c^{-1}(c+1)\big) \cdot \big(-(c+1)^{-1}\big) = 1, \]
so $r\big(1 + p^i + p^{2i}\big) \equiv 0 \mymod{p^d-1}$. Thus $r \equiv 0 \mymod{\big(p^d - 1\big)/\big(p^d - 1, \ p^{2i} + p^i + 1\big)}$. By Lemma \ref{lemma:divisibility} part \ref{div3} we have $\big(p^d - 1, \ p^{2i} + p^i + 1\big) = p^{2(d,i)} + p^{(d,i)} + 1$. Also $p^{3(d,i)} - 1 \nmid p^{2(d,i)} + p^{(d,i)} + 1$ but $p^{3(d,i)} - 1 \,\big|\, p^d - 1$, so $p^{(d,i)} - 1 \,\big|\, \big(p^d - 1\big)/\big(p^{2(d,i)} + p^{(d,i)} + 1\big)$. It follows that $r \equiv 0\bigmod{p^{(d,i)} - 1}$, so $c \in \big\langle \omega^{p^{(d,i)}-1} \big\rangle = \big\langle \omega^{p^i-1} \big\rangle$. Thus $c = b^{p^i - 1}$ for some $b \in \F^\#_q$. Since $c \in C_1(p,d,i)$, we have
	\[ b^{p^{2i}-p^i} = \big(b^{p^i-1}\big)^{p^i} = -b^{-(p^i-1)}\big(b^{p^i-1} + 1\big), \]
which yields $b + b^{p^i} + b^{p^{2i}} = 0$. Therefore $C_1(p,d,i) = \big\{ b^{p^i-1} \ \vline \ b \neq 0, \ b + b^{p^i} + b^{p^{2i}} = 0 \big\}$. To compute $|C_1(p,d,i)|$ let $\mathbb{E} = \Fix(\langle \tau^{3i}_p \rangle)$ and $\mathbb{K} = \Fix(\langle \tau^i_p \rangle)$. Observe that $b^{p^i-1} \in C_1(p,d,i)$ implies that $1 = \big(b^{p^i-1}\big)^{p^{2i} + p^i + 1} = b^{p^{3i} - 1}$, so $b \in \mathbb{E}$, and in this case $b + b^{p^i} + b^{p^{2i}} = \Tr_{\mathbb{E}/\mathbb{K}}(b)$. Note also that for any $b, b_0 \in \mathbb{E}$, we have $b^{p^i-1} = b_0^{p^i-1}$ if and only if $bb_0^{-1} \in \mathbb{K}$. Thus
	\[ |C_1(p,d,i)|
	= \frac{\left|\left\{ b \in \mathbb{E}^\# \ \vline \ \Tr_{\mathbb{E}/\mathbb{K}}(b) = 0 \right\}\right|}{\left|\mathbb{K}^\#\right|}
	= \frac{p^{2(d,i)} - 1}{p^{(d,i)} - 1}
	= p^{(d,i)} + 1, \]
which completes the proof of statement 2.

Assume that $\left|\tau^i_p\right|$ is odd. Then $(d,2i) = (d,i)$. Hence $c \in C_2(p,d,i)$ if and only if $c^{\tau^i_p} = -c(c+1)^{-1}$ and $c \in \Fix(\langle \tau^{2i}_p \rangle) = \Fix(\langle \tau^i_p \rangle)$. Equivalently $c = -c(c+1)^{-1} \neq 0$, which holds if and only if $c = -2$ and $p \neq 2$. This proves statement 3.

Finally, assume that $\left|\tau^i_p\right|$ is even. It is easy to check that $\big\{ b-1 \ \vline \ b \neq 1, \ b^{p^i + 1} = 1 \big\} \subseteq C_2(p,d,i)$. Let $c \in C_2(p,d,i)$ and $b := c + 1$. Then $b \neq 1$ and
	\[ b^{p^i+1} = \big(c^{p^i} + 1\big)(c + 1) = 1. \]
Thus $C_2(p,d,i) = \big\{ b-1 \ \vline \ b \neq 1, \ b^{p^i + 1} = 1 \big\}$. Applying Lemma \ref{lemma:divisibility} part \ref{div2}, we obtain
	\[ |C_2(p,d,i)| = \gcd\big(p^d - 1, \ p^i + 1\big) - 1 = p^{(d,i)} + 1 - 1 = p^{(d,i)}, \]
which completes the proof of statement 4.
\end{proof}

\begin{lemma} \label{lemma:C1-C2-obs}
Let $p$ be a prime, $d$, $i$, and $j$ be positive integers, and $q = p^d$. Let $\tau_p$ be the Frobenius map on $\F_q$ and let $C_1(p,d,i)$ and $C_2(p,d,i)$ be as in (\ref{equation:C1(p,d,i)}) and (\ref{equation:C2(p,d,i)}), respectively.
	\begin{enumerate}[1.]
	\item If $\left|\tau^i_p\right| \not\equiv 0 \mymod{3}$, or $\left|\tau^i_p\right| \equiv 0 \mymod{3}$ and $j/(i,j) \not\equiv 0 \mymod{3}$, then
		\[ C_1(p,d,i) \cap \Fix(\langle \tau^j_p \rangle) =
			\left\{\begin{aligned}
			&\left\{ c \ \vline \ |c| = 3 \right\} &&\text{if $p^{(d,i)} \equiv p^{(d,j)} \equiv 1 \mymod{3}$;} \\
			&\{1\} &&\text{if $p = 3$;} \\
			&\varnothing &&\text{otherwise.}
			\end{aligned}\right.\]
	\item If $\left|\tau^i_p\right| \equiv 0 \mymod{3}$ and $j/(i,j) \equiv 0 \mymod{3}$ then
		\[ C_1(p,d,i) \cap \Fix(\langle \tau^j_p \rangle) =
			\left\{\begin{aligned}
			&C_1(p,d,(i,j)) &&\text{if $i/(i,j) \equiv 1 \mymod{3}$;} \\
			&\left\{ c \ \vline \ c^{-1} \in C_1(p,d,(i,j)) \right\} &&\text{if $i/(i,j) \equiv 2 \mymod{3}$.}
			\end{aligned}\right. \]
	\item $C_1(p,d,i) \cap C_2(p,d,j) =
		\left\{\begin{aligned}
		&\{ c \ \vline \ |c| = 3 \} &&\text{if $p \equiv 2 \mymod{3}$, $\left|\tau^j_p\right|$ and $i$ are even, and $j$ is odd;} \\
		&\{1\} &&\text{if $p = 3$;} \\
		&\varnothing &&\text{otherwise.}
		\end{aligned}\right.$
	\end{enumerate}
\end{lemma}

\begin{proof}

We first show part 1. If $\left|\tau^i_p\right| \not\equiv 0 \mymod{3}$ then the result follows easily from part 1 of Lemma \ref{lemma:C1-C2}. Assume that $\left|\tau^i_p\right| \equiv 0 \mymod{3}$ and $j/(i,j) \not\equiv 0 \mymod{3}$. Then $c \in C_1(p,d,i) \cap \Fix(\langle \tau^j_p \rangle)$ implies that $c^{p^{2i} + p^i + 1} = c^{p^j - 1} = 1$, so that $|c|$ divides $\big(p^j - 1, \ p^{2i} + p^i + 1\big)$. If $p = 3$ then $\big(p^j - 1, \ p^{2i} + p^i + 1\big) = 1$ by Lemma \ref{lemma:divisibility} part \ref{div3}, implying that $c = 1$, which is indeed in $C_1(p,d,i)$. If $p \equiv 2 \mymod{3}$ and either $i$ or $j$ is odd then also $\big(p^j - 1, \ p^{2i} + p^i + 1\big) = 1$ and $c = 1$, but $1 \notin C_1(p,d,i)$. If $p \equiv 1 \mymod{3}$, or $p \equiv 2 \mymod{3}$ and both $i$ and $j$ are even then $\big(p^j - 1, \ p^{2i} + p^i + 1\big) = 3$; since $1 \notin C_1(p,d,i)$ this implies that $|c| = 3$. Clearly $c \in C_1(p,d,i)$ if and only if $p^{(d,i)} \equiv 1 \mymod{3}$. This proves statement 1.

Assume that $\left|\tau^i_p\right| \equiv 0 \mymod{3}$ and $j/(i,j) \equiv 0 \mymod{3}$. Then $i/(i,j) \not\equiv 0 \mymod{3}$. \emph{Case 1.} Suppose that $i/(i,j) \equiv 1 \mymod{3}$. Recall that any $c \in C_1(p,d,i)$ is fixed by $\tau^{3i}_p$. Then $c \in C_1(p,d,i) \cap \Fix(\langle \tau^j_p \rangle)$ implies that $|c| \,\big|\, \big(p^{3i} - 1, \ p^j - 1\big) = p^{3(i,j)} - 1$. Hence
	\[ c^{p^{(i,j)}} = c^{p^i} = -c^{-1}(c+1), \]
and thus $c \in C_1(p,d,(i,j))$. Conversely, any $c \in C_1(p,d,(i,j))$ is fixed by $\tau^{3(i,j)}_p$, so
	\[ c^{p^i} = c^{p^{(i,j)}} = -c^{-1}(c+1). \]
Also $c$ is fixed by $\tau^j_p$ since $3(i,j) \mid j$. Thus $c \in C_1(p,d,i) \cap \Fix(\langle \tau^j_p \rangle)$, which proves the first part of statement 2. \emph{Case 2.} Suppose that $i/(i,j) \equiv 2 \mymod{3}$. As in Case 1, $c \in C_1(p,d,j) \cap \Fix(\langle \tau^j_p \rangle)$ implies that $\big|c^{-1}\big| = |c| \,\big|\, p^{3(i,j)} - 1$. It is easy to show that $c^{-1} \in C_1(p,d,2i)$. So $c^{-1} \in C_1(p,d,2i) \cap \Fix\big(\big\langle \tau^{3(i,j)}_p \big\rangle\big)$, where $2i/(2i,3(i,j)) = 2i/(i,j) \equiv 1 \mymod{3}$. It follows from Case 1 that $c^{-1} \in C_1(p,d,(2i,3(i,j))) = C_1(p,d,(i,j))$. Conversely, if $c^{-1} \in C_1(p,d,(i,j))$ then $c \in C_1(p,d,2(i,j))$. Furthermore $c^{-1}$ is fixed by $\tau^{3(i,j)}_p$, and hence so is $c$. Thus
	\[ c^{p^i} = c^{p^{2(i,j)}} = -c^{-1}(c+1), \]
so $c \in C_1(p,d,i)$; also $c^{p^j} = c$. Therefore $c \in C_1(p,d,i) \cap \Fix(\langle \tau^j_p \rangle)$, which proves the second part of statement 2.

For statement 3 we need the following result: \emph{Claim.} If $\big|\tau^j_p\big|$ is even, then $C_2(p,d,j)$ has an element of order $3$ if and only if $p \equiv 2 \mymod{3}$ and $j$ is odd. Indeed, if $c \in C_2(p,d,i)$ with $|c| = 3$ then $p \neq 3$, $c = b - 1$ for some $b \in \F^\#_q$ with $b^{p^i + 1} = 1$, and $c^2 + c + 1 = 0$. Hence $b^2 - b + 1 = 0$, so that $b^3 = -1$ and $|b| \mid \big(p^i + 1, \ 6)$. Now if $p \equiv 1 \mymod{3}$ or $i$ is even then $\big(p^i + 1, \ 6) = 2$ yielding $b = -1$ and $p = 3$, contradiction. Thus $p \equiv 2 \mymod{3}$ and $i$ is odd. Conversely assume that $p \equiv 2 \mymod{3}$ and $i$ is odd. If $p = 2$ take $c = b - 1$ where $|b| = 3$, and if $p > 2$ take $c = b - 1$ where $|b| = 6$. In each case it is easy to verify that $b \neq 1$, $b^{p^i + 1} = 1$, and $c^3 = 1$. This proves part 3.

We now prove statement 3. For brevity let $\overline{C} = C_1(p,d,i) \cap C_2(p,d,j)$. \emph{Case 1.} Suppose that $\big|\tau^j_p\big|$ is odd. Then by Lemma \ref{lemma:C1-C2} (3.) we get $\overline{C} = \varnothing$ if $p = 2$, and $\overline{C} = C_1(p,d,i) \cap \{-2\}$ if $p \geq 3$; in the latter case $\overline{C} = \{-2\} = \{1\}$ if $p = 3$ and $\overline{C} = \varnothing$ otherwise. \emph{Case 2.} Suppose that $\big|\tau^j_p\big|$ is even and $\big|\tau^i_p\big| \not\equiv 0 \mymod{3}$. Then by Lemma \ref{lemma:C1-C2} (1.) and the claim above,
	\begin{align*}
	\overline{C} 
	&= \left\{\begin{aligned}
			&\{ c \ | \ |c| = 3 \} \cap C_2(p,d,j) &&\text{if $p^{(d,i)} \equiv 1 \mymod{3}$}; \\
			&\{1\} &&\text{if $p = 3$}; \\
			&\varnothing &&\text{otherwise}
			\end{aligned}\right. \\
	&= \left\{\begin{aligned}
			&\{ c \ | \ |c| = 3 \} &&\text{if $p^{(d,i)} \equiv 1 \mymod{3}$, $p \equiv 2 \mymod{3}$, and $j$ is odd}; \\
			&\{1\} &&\text{if $p = 3$}; \\
			&\varnothing &&\text{otherwise}
			\end{aligned}\right.
	\end{align*}
\emph{Case 3.} Suppose that $\big|\tau^j_p\big|$ is even and $\big|\tau^i_p\big| \equiv 0 \mymod{3}$. Then $c \in \overline{C}$ implies that $|c| \,\big|\, \big(p^{2i} + p^i + 1, \ p^{2j} - 1, \ p^i + p^j\big)$. \emph{Subcase 3.1.} Assume that either $i$ is odd, or $i$ and $j$ are even and $i/(i,j)$ is odd. Then $i/(i, \ |i-j|) = i/(i,j)$ is odd, and applying Lemma \ref{lemma:divisibility} (\ref{div4}.) we get $\big(p^i + p^j, \ p^{2i} + p^i + 1\big) = \big(p^{|i-j|} + 1, \ p^{2i} + p^i + 1\big) = 1$. Hence $\overline{C} = \{1\}$ if $p = 3$ and $\overline{C} = \varnothing$ otherwise. \emph{Subcase 3.2.} Assume that $i$ is even, and either $j$ is odd or $i/(i,j)$ is even. Then $j/(i,j)$ is odd, implying that $|i-j|/(j, \ |i-j|)$ is odd and $(2j, \ |i-j|) = (i,j)$. By Lemma \ref{lemma:divisibility} (\ref{div2}.) we have $\big(p^{2j} - 1, \ p^i + p^j\big) = \big(p^{2j} - 1, \ p^{|i-j|} + 1\big) = p^{(i,j)} + 1$. Hence $|c| \,\big|\, \big(p^{(i,j)} + 1, \ p^{2i} + p^i + 1\big)$. Since $(i,j)/((i,j), \ i) = 1 \not\equiv 0 \mymod{3}$ it follows from Lemma \ref{lemma:divisibility} (\ref{div4}.) that $|c| = 3$ if $p \equiv 2\mymod{3}$ and $j$ is odd, and $c = 1$ otherwise. Thus, from the claim, we have
	\[ \overline{C} =
			\left\{\begin{aligned}
			&\{ c \ | \ |c| = 3 \} &&\text{if $p \equiv 2 \mymod{3}$ and $j$ is odd}; \\
			&\{1\} &&\text{if $p = 3$}; \\
			&\varnothing &&\text{otherwise}.
			\end{aligned}\right. \]
Statement 3 follows.
\end{proof}
\section{Arc-transitive latin square graphs from groups} \label{section:arctranslsg}

Throughout this section we assume that $H$ is a finite group, $\Gamma = \lsg(H)$ is the latin square graph of $H$, and $\mathcal{G}$ is the autoparatopism group of the Cayley table of $H$. Recall from Subsection \ref{subsection:lsg} that $\mathcal{G} \cong \left(H^2 \rtimes \Aut(H)\right) \Sym(3)$. Let $\lambda_a$ and $\rho_b$ be as in (\ref{equation:regularaction}), and let $T$ be the autotopism subgroup of $\Gamma$, that is, $T = \left\{ (\lambda_a, \rho_b, \lambda_a\rho_b) \ \vline \ a, b \in H \right\}$. For each $\sigma \in \Aut(H)$ and $A \leq \Aut(H)$, identify $\sigma$ with $(\sigma, \sigma, \sigma) \in \Sym(H)^3$ and $A$ with $\left\{ (\sigma,\sigma,\sigma) \ | \ \sigma \in A \right\}$. For any $G \leq \mathcal{G}$, denote the point stabilizer in $G$ of any $v \in \V(\Gamma)$ by $G_v$.

We begin by finding necessary conditions in order for the group $\mathcal{G}$ to be arc-transitive on $\Gamma$. Let $x$ and $y$ be as in Table \ref{table:autoparatopisms} It is easy to see that $\mathcal{G}_\mathbf{1} = \langle x, \, y, \, \Aut(H) \rangle \cong \Aut(H) \times \Sym(3)$. Hence:


\begin{theorem} \label{theorem:arctrans-elemab}
Let $H$ be a finite group, $\Gamma = \lsg(H)$, and $\mathcal{G}$ the autoparatopism group of the Cayley table of $H$. Then $\Gamma$ is $\mathcal{G}$-arc-transitive if and only if $H$ is elementary abelian.
\end{theorem}

\begin{proof}
Let $\Gamma(\mathbf{1})$ denote the set of all neighbors in $\Gamma$ of the vertex $\mathbf{1}$. Then $\Gamma(\mathbf{1}) = \big\{ (1, h, h), \, (h, 1, h),$ $(h, h^{-1}, 1) \ \big| \ h \in H^\# \big\}$. Since $\Gamma$ is $\mathcal{G}$-vertex-transitive, it is $\mathcal{G}$-arc-transitive if and only if $\mathcal{G}_{\mathbf{1}}$ is transitive on $\Gamma(\mathbf{1})$. Let $a \in H^\#$ and $v = (1, a, a)$. Then $v \in \Gamma(\mathbf{1})$ and $v^{\mathcal{G}_{\mathbf{1}}} = \big\{ (1, b, b), \, \left(b^{-1}, 1, b^{-1}\right),$ $\left(b, b^{-1}, 1\right) \ \big| \ b \in a^{\Aut(H)} \big\}$. If $H$ is elementary abelian then $\Aut(H)$ is transitive on $H^\#$, and thus $v^{\mathcal{G}_{\mathbf{1}}} = \Gamma(\mathbf{1})$. Conversely, suppose that $v^{\mathcal{G}_{\mathbf{1}}} = \Gamma(\mathbf{1})$. Then for any $c \in H^\#$ and $w = (1, c, c)$ there exists $\sigma \in \Aut(H)$ and $z \in \langle x, y \rangle$ such that $w = v^{\sigma z}$. Now $v^{\sigma z} \in \big\{ (1, b, b), \left(b^{-1}, 1, b^{-1}\right), \left(b, b^{-1}, 1\right) \ \big|$ $b = a^\sigma \big\}$, so $w = \left(1, b, b\right)$ and $c = b = a^\sigma$. Thus $\Aut(H)$ is transitive on $H^\#$ and $H$ is elementary abelian. Therefore $\mathcal{G}_{\mathbf{1}}$ is transitive on $\Gamma(\mathbf{1})$ if and only if $H$ is elementary abelian, and the result follows.
\end{proof}

In view of Theorem \ref{theorem:arctrans-elemab}, assume for the remainder of this paper that $H$ is elementary abelian, and shift to additive notation as follows:

\begin{notation}
Let $H = C^d_p$ for some prime $p$ and positive integer $d$. Denote the inverse of $h \in H$ by $-h$, the identity element of $H$ by $0_H$, and the vertex $\big(0_H, 0_H, 0_H\big)$ by $\mathbf{0}$. Then for any $a, b \in H$ we have $\lambda_a = \rho_{-a}$ and $\lambda_a\rho_b = \rho_{-a+b}$, so that
	\begin{equation} \label{equation:T}
	T = \left\{ (\rho_a, \rho_b, \rho_{a+b}) \ \vline \ a, b \in H \right\}.
	\end{equation}
Furthermore we can identify $H$ with a vector space $U$ of dimension $n$ over $\F_q$, where $n$ is a divisor of $d$ and $q = p^{d/n}$, and $H^2$ as $V = U \oplus U$ whose elements are written as ordered pairs in $U$. Following \cite{AGP-quotcomp} we denote some special subsets of $V$ as follows:
	\begin{align}
	V_\infty &= \left\{ (0_U,u) \ | \ u \in U \right\} = \{0_U\} \oplus U, \label{equation:V_inf} \\
	V_c &= \left\{ (u,cu) \ | \ u \in U \right\} \; \text{for any $c \in \F_q$}, \label{equation:V_c} \\
	V_{\text{$q$-ind}} &= \left\{ (u,v) \ | \ \{u,v\} \text{ linearly independent in $U$ } \right\} \ \text{(provided $n \geq 2$)} \label{equation:V_ind}
	\end{align}
Note that $V_0 = U \oplus \{0_U\}$ and ${\rm diag}(U \oplus U) = V_1$.
\end{notation}

We shall also frequently identify the elements of the group $\langle x, y \rangle$ with the induced permutation on coordinates. Hence $x$ is identified with $(1\,2\,3)$ and $y$ with $(1\,3)$.

With the notation above, $\Gamma \cong \Cay(V,S)$ where $S = \Gamma(\mathbf{0}) = V^\#_\infty \cup V^\#_0 \cup V^\#_{-1}$. Also $T$ is the translation group on $V$ and $\mathcal{G}_\mathbf{0} = \GL_d(p) \times \Sym(3)$. Any arc-transitive group $G \leq \mathcal{G}$ is of the form $G = TG_\mathbf{0}$ where $G_\mathbf{0}$ is transitive on $\Gamma(\mathbf{0})$. For the rest of this chapter we determine all transitive groups $G_\mathbf{0}$, beginning with some necessary conditions for transitivity on $\Gamma(\mathbf{0})$, which are given in Lemma \ref{lemma:arctrans-proj}.

\begin{lemma} \label{lemma:arctrans-proj}
Let $H$ be an elementary abelian group, $\Gamma = \lsg(H)$, $G \leq \mathcal{G} \leq \Aut(\Gamma)$, and $\pi : \mathcal{G}_{\mathbf{0}} \rightarrow \Aut(H)$ and $\theta : \mathcal{G}_{\mathbf{0}} \rightarrow \langle x, y \rangle$ be the natural projections. If $\Gamma$ is $G$-arc-transitive then $G = TG_\mathbf{0}$, where $\pi(G_\mathbf{0})$ is transitive on $H^\#$ and $\theta(G_\mathbf{0})$ is transitive on $\{1,2,3\}$.
\end{lemma}

\begin{proof}
Assume that $\Gamma$ is $G$-arc-transitive. By the remarks above $G = TG_\mathbf{0}$ where $G_\mathbf{0}$ is transitive on $\Gamma(\mathbf{0})$. Let $h \in H^\#$ and consider the vertex $v = (0, h, h)$. Let $h_1 \in H^\#$. Then there exists $g \in G_\mathbf{0}$ such that $(0, h_1, h_1) = v^g$. Write $g$ as $\pi(g)\theta(g)$. If $h^{\pi(g)} = h_2$ then
	\[ v^g = (0, h_2, h_2)^{\theta(g)} \in \left\{ (0, h_2, h_2), (-h_2, 0, -h_2), (h_2, -h_2, 0) \right\}, \]
which implies that $(0, h_1, h_1) = (0, h_2, h_2)$. Thus $h_1 = h_2 = h^{\pi(g)}$, and $\pi(G_\mathbf{0})$ is transitive on $H^\#$. Also $(h, 0, h) = v^{g_1}$ for some $g_1 \in G_\mathbf{0}$. In this case $h^{\pi(g_1)} = h$ so that $v^{g_1} = v^{\theta(g)}$, and hence $1^{\theta(g_1)} = 2$. Similarly, $(h, -h, 0) = v^{g_2}$ for some $g_2 \in G_\mathbf{0}$, which yields $1^{\theta(g_2)} = 3$. Therefore $\theta(G_\mathbf{0})$ is transitive on $\{1,2,3\}$.
\end{proof}

It follows from Lemma \ref{lemma:arctrans-proj} that the projections $\pi(G_\mathbf{0})$ are finite linear groups. The transitive finite linear groups were classified by C. Hering in \cite{Hering}; we state  this result in Theorem \ref{theorem:Hering}.

\begin{theorem} \cite{Hering} \label{theorem:Hering}
Let $U$ be a vector space of dimension $d$ over the prime field of order $p$, and let $X \leq \GL_d(p)$ be transitive on $U^\#$. Then $X$ is one of the types given in Table \ref{table:translin}, or, setting $q = p^{d/n}$ for some divisor $n$ of $d$, the group $X$ belongs to one of the following classes:
	\begin{enumerate}[1.] \itemsep0pt
	\item $X \leq \GamL_1(q)$, $n = 1$;
	\item $X \unrhd \SL_n(q)$, $n \geq 2$;
	\item $X \unrhd \Sp_n(q)$, $d$ and $n$ even;
	\item $X \unrhd G_2(q)$, $n = 6$ and $p = 2$.
	\end{enumerate}
\end{theorem}

\begin{table}
\begin{center}
\begin{tabular}{lcl}
\hline
$p$ & $d$ & $X$ \\
\hline\hline
5, 7, 11, 23 & 2 & $X \leq N_{\GL_d(p)}(Q_8)$ \\
11, 19, 29, 59 & 2 & $X \trianglerighteq \SL_2(5)$ \\
3 & 4 & $\SL_2(5) \trianglelefteq X \leq \GamL_2(9)$ \\
3 & 4 & $X \leq N_{\GL_4(3)}(D_8 \circ Q_8)$ \\
2 & 4 & $\Alt(6)$ \\
2 & 4 & $\Alt(7)$ \\
3 & 6 & $\SL_2(13)$ \\
\hline
\end{tabular}
\caption{Sporadic transitive finite linear groups}
\label{table:translin}
\end{center}
\end{table}

By Lemma \ref{lemma:arctrans-proj}, $\theta(G_\mathbf{0})$ is necessarily either $\langle x \rangle$ or $\langle x, y \rangle$. The subgroups of $\GL_d(p) \times \Sym(3)$ are determined by applying Goursat's Lemma (see, for instance, \cite[Exercise 5, p. 75]{Lang}), which describes the subgroups of a direct product of two groups. From this and Theorem \ref{theorem:Hering} we deduce the following:

\begin{theorem} \label{theorem:G_0}
Let $H$ be an elementary abelian group, $\Gamma = \lsg(H)$, $G \leq \mathcal{G} \leq \Aut(\Gamma)$, and $\pi : \mathcal{G}_{\mathbf{0}} \rightarrow \Aut(H)$ be the natural projection. If $\Gamma$ is $G$-arc-transitive, then $G = TG_\mathbf{0}$ where $\pi(G_\mathbf{0})$ is one of the groups in Theorem \ref{theorem:Hering}. Furthermore, for suitable $K \unlhd \pi(G_\mathbf{0})$ and $g, h \in \Aut(H) \setminus K$, the group $G_\mathbf{0}$ is one of the types described in Table \ref{table:G_0}.
\end{theorem}

	\begin{center}
	\begin{table}
	\begin{tabular}{rllll}
	\hline
	& $G_\mathbf{0}$ & $\pi(G_\mathbf{0})$ & Conditions on cosets of $K$ & Transitive on $H^\#$ \\
	\hline\hline
	\small\textsf{1} & $\langle K, x \rangle$ & $K$ &  & $K$ \\
	\small\textsf{2} & $\langle K, gx \rangle$ & $\langle K, g \rangle$ & $Kg^3 = K$ & $K$ \\
	\small\textsf{3} & $\langle K, x, y \rangle$ & $K$ &  & $K$ \\
	\small\textsf{4} & $\langle K, x, hy \rangle$ & $\langle K, h \rangle$ & $Kh^2 = K$ & $\langle K, h \rangle$ \\
	\small\textsf{5} & $\langle K, gx, hy \rangle$ & $\langle K, g, h \rangle$ & $Kg^3 = Kh^2 = K$, $Kgh = Khg^2$ & $\langle K, h \rangle$ \\
	\hline
	\end{tabular}
	\caption{Possible groups $G_\mathbf{0}$} \label{table:G_0}
	\end{table}
	\end{center}

\begin{proof}
The first three columns of Table \ref{table:G_0} follow immediately from Goursat's Lemma \cite[Exercise 5, p. 75]{Lang}. Clearly Lemma \ref{lemma:arctrans-proj} implies that for $G_\mathbf{0}$ in lines 1 or 3 the group $K$ must be transitive on $H^\#$. Let $a \in H^\#$. For $G_\mathbf{0}$ in line 2 we have $(a,0_U)^{G_\mathbf{0}} = (a,0_U)^K \cup (-a^g,a^g)^K \cup \big(0_U,-a^{g^2}\big)^K$, so $(a,0_U)^{G_\mathbf{0}} = \Gamma(\mathbf{0})$ if and only if $(a,0_U)^K = V^\#_0$ and equivalently $K$ is transitive on $U^\#$. For $G_\mathbf{0}$ in  line 4 we have $(a,0_U)^{G_\mathbf{0}} = (a,0_U)^{\langle K, h \rangle} \cup (-a,a)^{\langle K, h \rangle} \cup (0_U,-a)^{\langle K, h \rangle}$, and for $G_\mathbf{0}$ in line 5 we have $(a,0_U)^{G_\mathbf{0}} = (a,0_U)^{\langle K, h \rangle} \cup \big(-a^{hg},a^{hg}\big)^{\langle K, h \rangle} \cup \big(0_U,-a^{hg^2}\big)^{\langle K, h \rangle}$. In both cases $(a,0_U)^{G_\mathbf{0}} = \Gamma(\mathbf{0})$ if and only if $\langle K, h \rangle$ is transitive on $U^\#$.
\end{proof}

%

The groups in lines 4 and 5 of Table \ref{table:G_0} correspond to subgroups of index 2 of transitive finite linear groups. In view of this, we list in Theorem \ref{theorem:index2} the transitive finite linear groups which contain an intransitive subgroup of index $2$.

\begin{theorem} \label{theorem:index2}
Let $U$ be a vector space of dimension $d$ over the prime field of order $p$ and let $X \leq \GL_d(p)$ be transitive on $U^\#$. Suppose that $X$ has a normal subgroup of index $2$ which is intransitive on $U^\#$. Then either $X \unlhd \GamL_1(q)$ where $q = p^d$, or $X$ is a group belonging to Table \ref{table:translin} and in particular, using the notation of \cite[Table 2]{AGP-quotcomp}, $X$ is one of the following groups: $H(2)$, $H(5)$, $H(6)$, $H(7)$, $H(10)$, $A(2)$, $A(3)$, $A(4)$, $N_{\GamL_2(9)}(\SL_2(5))$, and $B(1)$.
\end{theorem}

\begin{proof}
Let $q = p^{d/n}$ for some divisor $n$ of $d$ with $n \geq 2$, and suppose that $X$ has a normal subgroup $K$ of index $2$ which is intransitive on $U^\#$. If $G_2(q) \unlhd X$ ($p = 2$ and $n = 6$) then $K \unrhd G_2(q)$, so $K$ is transitive, contradiction. Assume that $L \unlhd X$ and $L \nleq K$, where $L = \SL_n(q)$ ($n > 2$ or $q > 3$) or $L = \Sp_n(q)$ ($n > 4$ or $q > 2$). It can be deduced from Iwasawa's criterion \cite[Theorem 1.2]{Taylor} that $K \cap L \unlhd Z(L)$ or $L \unlhd K$. Since $K$ is intransitive it must be that $K \cap L \unlhd Z(L)$. Also $KL = X$, so that $|X : K| = |KL : K| = |L : K \cap L| \geq |L : Z(L)| > 2$, contradiction. If $X$ is one of $\SL_2(2)$, $\SL_2(3)$, and $\Sp_4(2)$, it is easily verified that $X$ has no subgroup $K$ satisfying the above conditions. Thus $X$ is one of the groups in Table \ref{table:translin}, and in this case the specific groups $X$ are found using \textsc{Magma} \cite{Magma}.
\end{proof}

\section{Main results} \label{section:quotcomplsg}

We collect in Hypothesis \ref{hypothesis} the conditions that we assume for the remainder of this paper.

\begin{hypothesis} \label{hypothesis}
For a prime $p$ and positive integer $d$, the group $H = C_p^d$ is viewed as a vector space $U$ of dimension $n$ over the finite field $\F_q$, where $n$ is a divisor of $d$ and $q = p^{d/n}$. The group $H^2$ is viewed as $V = U \oplus U$ with elements written as ordered pairs. The group $T$ is as in (\ref{equation:T}) and $G = T \rtimes G_\mathbf{0}$ where $G_\mathbf{0}$ is one of the types described in Table \ref{table:G_0}.
\end{hypothesis}

In addition to the notation defined in (\ref{equation:V_inf}) -- (\ref{equation:alpha}), for any subspace $W$ of $V$ we denote by $T_W$ the group of translations of $V$ by elements of $W$, that is,
	\begin{equation} \label{equation:T_W}
	T_W := \left\{ (\rho_a, \rho_b, \rho_{a+b}) \ | \ (a,b) \in W \right\}.
	\end{equation}
Also for any $c \in \F^\#_q$ we denote by $|c|$ the multiplicative order of $c$ in $\F^\#_q$. Recall that $\GamL_n(q) \lesssim \GL_d(p)$ and $\GamL_n(q) = \GL_n(q) \rtimes \Aut(\F_q)$; let
	\begin{equation} \label{equation:alpha}
	\alpha : \GamL_n(q) \rightarrow \Aut(\F_q)
	\end{equation}
be the natural projection map.

Since $H$ is abelian, it can be shown that for any $z \in \langle x, y \rangle$ and $a, b \in H$,
	\[ z^{-1} (\rho_a, \rho_b, \rho_{a+b}) z = (\rho_{a'}, \rho_{b'}, \rho_{a'+b'}) \]
where $(a', b', a'+b') = (a, b, a+b)^z$. Also the correspondence $(a,b) \mapsto (a,b)^z := (a',b')$ defines an action of $\langle x, y \rangle$ on $V$ induced by the action on $\V(\Gamma)$, which together with the componentwise action of $\GL(U)$ yields an action of $G_\mathbf{0}$ on $V$. It follows that a normal subgroup of $G$ corresponds to a subspace of $V$ generated by a union of orbits of $G_\mathbf{0}$ under this action. We are interested in finding minimal normal subgroups of $G$ which are intransitive on $V$; these correspond to $G_\mathbf{0}$-orbits in $V^\#$ which generate proper subspaces of $V$.

In Lemma \ref{lemma:orbits} we give some small results about subspaces generated by $G_\mathbf{0}$-orbits. These will be used in the proof of Lemma \ref{lemma:subspaces-V_dep}.

\begin{lemma} \label{lemma:orbits}
Let $U$ be a vector space of dimension $n$ over $\F_q$, $V = U \oplus U$, and $K \leq \GamL_n(q)$.
	\begin{enumerate}[1.]
	\item Suppose that $K$ is transitive on $U^\#$ and let $(a_1,b_1), (a_2,b_2) \in V^\#$. If $(a_1,b_1)^K \neq (a_2,b_2)^K$, then $\big\langle (a_1,b_1)^K, \ (a_2,b_2)^K \big\rangle = V$.
	\item Suppose that $K \leq X \leq \GamL_n(q)$, with $X$ transitive on $U^\#$ and $|X : K| \leq 2$. If $a \in U^\#$ and $c \in \F^\#_q \setminus \{-1\}$, then $\big\langle (a,ca)^K \big\rangle = \bigcup_{c' \in c^{\alpha(K)}} V^\#_{c'}$.
	\end{enumerate}
\end{lemma}

\begin{proof}
Suppose that $K$ is transitive and $(a_1,b_1) \notin (a_2,b_2)^K$. Then there exists $b'_1 \neq b_1$ such that $(a_1,b'_1) \in (a_2,b_2)^K$, so that $(0, \ b_1-b'_1)^K = V^\#_\infty \subseteq W$. Similarly $V^\#_0 \subseteq W$. Therefore $W = \big\langle V^\#_0, V^\#_\infty \big\rangle = V$. This proves statement 2.

Let $a \in U^\#$ and $c \in \F^\#_q \setminus \{-1\}$. If $K$ is transitive on $U^\#$ then clearly $\big\langle (a,ca)^K \big\rangle = \bigcup_{c' \in c^{\alpha(K)}} V^\#_{c'}$. If $K$ is an intransitive subgroup of $X$ with $|X : K| = 2$ then $X$ is one of the groups listed in Lemma \ref{theorem:index2}; a case-by-case examination using \textsc{Magma} \cite{Magma} confirms that also $\big\langle (a,ca)^K \big\rangle = \bigcup_{c' \in c^{\alpha(K)}} V^\#_{c'}$. This proves statement 3.
\end{proof}

We now describe the subspaces generated by each $G_\mathbf{0}$-orbit in $V^\#$.

\begin{lemma} \label{lemma:subspaces-V_ind}
Assume Hypothesis \ref{hypothesis} with $n \geq 2$. Then $\langle S \rangle = V$ for any $G_\mathbf{0}$-orbit $S$ in $V_{\textnormal{$q$-ind}}$.
\end{lemma}

\begin{proof}
Let $(a,b) \in V_{\text{$q$-ind}}$ and $K \unlhd \pi(G_\mathbf{0})$ as in Theorem \ref{theorem:G_0}. If $K$ is one of the groups in Theorem \ref{theorem:Hering} (2.--4.) then $\big\langle (a,b)^K \big\rangle = V$ by \cite[Propositions 4.2 (1), 4.4 (2), and 4.8 (1--3)]{AGP-quotcomp}, and it follows that $\big\langle (a,b)^{G_\mathbf{0}} \big\rangle = V$ for any $G_\mathbf{0}$ in Theorem \ref{theorem:G_0}. If $K$ is one of the groups in Table \ref{table:translin} then with $d = 2$ then $\big\langle (a,b)^K \big\rangle = V$ by \cite[Lemma 4.9]{AGP-quotcomp}. For the remaining $K$, the result is verified using \textsc{Magma} \cite{Magma}.
\end{proof}

\begin{lemma} \label{lemma:subspaces-V_dep}
Assume Hypothesis \ref{hypothesis}. Let $c \in \F^\#_q \setminus \{-1\}$ and $S$ be a $G_\mathbf{0}$-orbit in $V_{\textnormal{$q$-dep}}$. Also let $\alpha(g) = \tau^i_p$, $\alpha(h) = \tau^j_p$, and $\alpha(K) = \langle \tau^\ell_p \rangle$, with $i, j, \ell \in \{1, \ldots, d\}$ and $m = \gcd(i,\ell)$. Then $\langle S \rangle = V_c$ if $S = (a,ca)^{G_\mathbf{0}}$ for some $a \in U^\#$, $c \in \F^\#_q \setminus \{-1\}$, and the conditions in Table \ref{table:V_c} are satisfied; otherwise, $\langle S \rangle = V$.
\end{lemma}

\begin{table}
\begin{tabular}{lll}
\hline
\small\textsf{r} & Conditions & $\big\langle (a,ca)^{G_\mathbf{0}} \big\rangle = V_c$ iff \\
\hline\hline
\small\textsf{1} & $p = 3$ & $c = 1$ \\
 & \begin{minipage}[t]{7.5cm}
	$p \equiv 1 \mymod{3}$, or $p \equiv 2 \mymod{3}$, $2 \mid d$, and $2 \mid \ell$ 
	\end{minipage} & $|c| = 3$ \\
\hline
\small\textsf{2} & $p = 3$, and either $3 \nmid |\tau^i_p|$ or $3 \nmid \ell/m$ & $c = 1$ \\
 & \begin{minipage}[t]{7.5cm}
	$p \equiv 1 \mymod{3}$, and either $3 \nmid |\tau^i_p|$ or $3 \nmid \ell/m$
	\end{minipage} & $|c| = 3$ \\
 & 
	\begin{minipage}[t]{7.5cm}
	$p \equiv 2 \mymod{3}$, $2 \mid d$, $2 \mid i$, $2 \mid \ell$, and either $3 \nmid |\tau^i_p|$ or $3 \nmid \ell/m$
	\end{minipage} & $|c| = 3$ \\
 & \begin{minipage}[t]{7.5cm}
	$3 \mid |\tau^i_p|$, $3 \mid \ell/m$, and $i/m \equiv 1\mymod{3}$
	\end{minipage} &
	\begin{minipage}[t]{5.5cm}
	$c = b^{p^m - 1}$, $b \in \F^\#_q$, $b + b^{p^m} + b^{p^{2m}} = 0$
	\end{minipage} \\
 & \begin{minipage}[t]{7.5cm}
	$3 \mid |\tau^i_p|$, $3 \mid \ell/m$, and $i/m \equiv 2\mymod{3}$
	\end{minipage} & 
	\begin{minipage}[t]{5.5cm}
	$c = b^{1 - p^m}$, $b \in \F^\#_q$, $b + b^{p^m} + b^{p^{2m}} = 0$
	\end{minipage} \\
\hline
\small\textsf{3} & $p = 3$ & $c = 1$ \\
\hline
\small\textsf{4} & $p = 3$ & $c = 1$ \\
 & \begin{minipage}[t]{7.5cm}
	$p \equiv 2\mymod{3}$, $2 \mid d$, $2 \nmid j$, and $2 \mid \ell$
	\end{minipage} & $|c| = 3$ \\
\hline
\small\textsf{5} & $p = 3$ & $c = 1$ \\
 & \begin{minipage}[t]{7.5cm}
	$p \equiv 2\mymod{3}$, $2 \mid |\tau^j_p|$, $2 \mid i$, $2 \nmid j$, and $2 \mid \ell$
	\end{minipage} & $|c| = 3$ \\
\hline
\end{tabular}
\caption{$\langle S \rangle$ for $G_\mathbf{0}$-orbits $S \subseteq V_{\text{$q$-dep}}$, with $G_\mathbf{0}$ as in Table \ref{table:G_0} line \small\textsf{r}} \label{table:V_c}
\end{table}

\begin{proof}
Assume throughout that $a \in U^\#$ and $c \in \F^\#_q \setminus \{-1\}$. For each $G_\mathbf{0}$ in Table \ref{table:G_0}, let $R$ be a complete set of coset representatives for $K$ in $G_\mathbf{0}$ and let $\Omega = \big\{ c_1^{-1}c_2 \ \big| \ (c_1a,c_2a) \in (a,ca)^R \big\}$. Hence in all cases
	\[ \big\langle (a,ca)^{G_\mathbf{0}} \big\rangle = \left\langle \bigcup_{c' \in \Delta} V^\#_{c'} \ \vline \ \Delta = \bigcup_{c'' \in \Omega} (c'')^{\alpha(K)} \right\rangle, \]
so it follows Lemma \ref{lemma:orbits} (2.) that $\big\langle (a,ca)^K \big\rangle = V_c$ if and only if $\Omega = \{c\}$ and $c \in \Fix(\alpha(K))$.

For $G_\mathbf{0}$ as in line 1, take $R = \big\{1, x, x^2\big\}$. Then $\Omega = \big\{ c, \, -(c+1)^{-1}, \, -c^{-1}(c+1) \big\} = \{c\}$ if and only if $c^2 + c + 1 = 0$. So $\big\langle (a,ca)^{G_\mathbf{0}} \big\rangle = V_c$ exactly when $c = 1$ and $p = 3$, or $|c| = 3$ and $p^{\gcd(d,\ell)} \equiv 1\mymod{3}$, which are precisely the conditions in row 1 of Table \ref{table:V_c}.

For $G_\mathbf{0}$ as in line 2, take $R = \big\{1, gx, g^2x^2 \big\}$. Then $\Omega = \big\{ c, \, (-(c+1)^{-1})^{\alpha(g)}, \, (-c^{-1}(c+1))^{\alpha(g^2)} \big\} = \{c\}$ if and only if $c \in C_1(p,d,i)$, where $C_1(p,d,i)$ is as in (\ref{equation:C1(p,d,i)}). So $\big\langle (a,ca)^{G_\mathbf{0}} \big\rangle = V_c$ exactly when $c \in C_1(p,d,i) \cap \Fix(\alpha(K))$. Applying Lemma \ref{lemma:C1-C2} (1.) and (2.) and Lemma \ref{lemma:C1-C2-obs} (1.) and (2.) yields the conditions in row 2 of Table \ref{table:V_c}.

For $G_\mathbf{0}$ as in line 3, take $R = \big\{1, x, x^2, y, xy, x^2y\big\}$. Then $\Omega = \big\{ c, \, -(c+1)^{-1}, \, -c^{-1}(c+1), \, -c(c+1)^{-1}, \, c^{-1}, \, -(c+1) \big\} = \{c\}$ if and only if $c^2 + c + 1 = 0$ and $c^2 = 1$. Equivalently, $c = 1$ and $p = 3$, which are the conditions in row 3 of Table \ref{table:V_c}.

For $G_\mathbf{0}$ as in line 4, take $R = \big\{1, x, x^2, hy, hxy, hx^2y\big\}$. Then $\Omega = \big\{ c, \, -(c+1)^{-1}, \, -c^{-1}(c+1), \, (-c(c+1)^{-1})^{\alpha(h)}, \, (c^{-1})^{\alpha(h)}, \, (-(c+1))^{\alpha(h)} \big\} = \{c\}$ if and only if $c^2 + c + 1 = 0$ and $c^{-p^j} = c$. Hence $\big\langle (a,ca)^{G_\mathbf{0}} \big\rangle = V_c$ exactly when either $c = 1$ and $p = 3$, or $|c| = 3$ and $p^{-j} \equiv p^{\gcd(d,\ell)} \equiv 1 \mymod{3}$. These are equivalent to the conditions in row 4 of Table \ref{table:V_c}.

For $G_\mathbf{0}$ as in line 5, take $R = \big\{1, gx, g^2x^2, hy, ghxy, g^2hx^2y\big\}$. Then $\Omega = \big\{ c, \, (-(c+1)^{-1})^{\alpha(g)},$ $(-c^{-1}(c+1))^{\alpha(g^2)}, \, (-c(c+1)^{-1})^{\alpha(h)}, \, (c^{-1})^{\alpha(gh)}, \, (-(c+1))^{\alpha(g^2h)} \big\} = \{c\}$ if and only if $c \in C_1(p,d,i) \cap C_2(p,d,j)$, where $C_1(p,d,i)$ and $C_2(p,d,j)$ are as in (\ref{equation:C1(p,d,i)}) and (\ref{equation:C2(p,d,i)}), respectively. Applying Lemma \ref{lemma:C1-C2-obs} (3.) yields the conditions in row 5 of Table \ref{table:V_c}.

That $\big\langle (a,ca)^{G_\mathbf{0}} \big\rangle = V$ for all cases not described in Table \ref{table:V_c} follows from \cite[Lemmas 2.6 and 2.7]{AGP-quotcomp}. This completes the proof.
\end{proof}

The main result in this section is Theorem \ref{theorem:maintheorem}.

\begin{theorem} \label{theorem:maintheorem}
Let $\Gamma = \lsg(H)$ and $G \leq \Aut(\Gamma)$, where $H$ is a finite group and $G$ is a subgroup of autoparatopisms of $\Gamma$. Then $\Gamma$ is $G$-arc-transitive if and only if $H$ is elementary abelian with additive identity $\mathbf{0}$, and the point stabilizer $G_\mathbf{0}$ is one of the groups in Table \ref{table:G_0}. Furthermore, with the notation of Hypothesis \ref{hypothesis} and Lemma \ref{lemma:subspaces-V_dep}, and with $T_W$ as in (\ref{equation:T_W}), one of the following holds:
	\begin{enumerate}[1.]
	\item The graph $\Gamma$ is $G$-quotient-complete with exactly one nontrivial complete $G$-normal quotient if and only if $p = 3$ and either
		\begin{enumerate}[a.]
		\item $G_\mathbf{0}$ is as in Table \ref{table:G_0} lines 1, 3, 4, or 5; or
		\item $G_\mathbf{0}$ is as in Table \ref{table:G_0} line 2, and either $3 \nmid |\tau^i_p|$ or $3 \nmid \ell/m$.
		\end{enumerate}
	The unique nontrivial complete normal quotient corresponds to the normal subgroup $T_{V_1}$.
	\item The graph $\Gamma$ is $G$-quotient-complete with exactly two nontrivial complete $G$-normal quotients if and only if one of the following holds:
		\begin{enumerate}[a.]
		\item $G_\mathbf{0}$ as in Table \ref{table:G_0} line 1, and either $p \equiv 1 \mymod{3}$, or $p \equiv 2 \mymod{3}$ and both $d$ and $\ell$ are even;
		\item $G_\mathbf{0}$ as in Table \ref{table:G_0} line 3; either $p \equiv 1 \mymod{3}$, or $p \equiv 2 \mymod{3}$ with $d$, $i$, and $\ell$ even; and either $3 \nmid |\tau^i_p|$ or $3 \nmid \ell/m$;
		\item $G_\mathbf{0}$ as in Table \ref{table:G_0} line 4, $p \equiv 2 \mymod{3}$, $j$ odd, and both $d$ and $\ell$ even;
		\item $G_\mathbf{0}$ as in Table \ref{table:G_0} line 5, $p \equiv 2\mymod{3}$, $|\alpha(h)|$ even, $i$ even, and $j$ odd.
		\end{enumerate}
	The two nontrivial complete normal quotients correspond to the normal subgroups $T_{V_c}$ and $T_{V_{c^{-1}}}$ where $|c| = 3$.
	\item The graph $\Gamma$ is $G$-arc-transitive and $G$-quotient-complete with exactly $p^{\gcd(d,m)} + 1$ nontrivial complete $G$-normal quotients if and only if $G_\mathbf{0}$ is as in Table \ref{table:G_0} line 2, $3 \mid |\alpha(g)|$, and $3 \mid \ell/m$. The nontrivial complete normal quotients correspond to the normal subgroups $T_{V_c}$ where
		\[ c = \left\{ \begin{aligned} &b^{p^m - 1} &&\text{if $i/m \equiv 1 \mymod{3}$} \\ &b^{1 - p^m} &&\text{if $i/m \equiv 2 \mymod{3}$} \end{aligned} \right. \]
for some $b \in \F^\#_q$ satisfying $b + b^{p^m} + b^{p^{2m}} = 0$.
	\item For all other cases, $\Gamma$ is $G$-vertex-quasiprimitive, i.e., all nontrivial normal subgroups of $G$ act transitively on the vertices of $\Gamma$.
	\end{enumerate}
\end{theorem}

\begin{proof}
That $G_\mathbf{0}$ is as in Table \ref{table:G_0} follows from Theorem \ref{theorem:G_0}. Suppose that the conditions in the second column of Table \ref{table:V_c} are satisfied for each respective $G_\mathbf{0}$. It follows immediately from Lemmas \ref{lemma:subspaces-V_ind} and \ref{lemma:subspaces-V_dep} that the minimal normal subgroups of $G$ are $T_{V_c}$ with $c$ as given in Table \ref{table:V_c}. Since $c \neq 0$, the elements in $T_{V_0}$ form a complete set of coset representatives for $T_{V_c}$ in $G$. Hence $\Gamma \cong K_q$, and statements 1--3 follow. Otherwise the minimal normal subgroup of $G$ is $T$, so that $\Gamma$ is $G$-vertex-quasiprimitive. This completes the proof.
\end{proof}

Observe that if $\Gamma$ has $k \geq 3$ nontrivial complete $G$-normal quotients then $\Gamma$ is as in Theorem \ref{theorem:maintheorem}, and in particular $\Gamma$ has order $p^{2d}$, $k = p^r + 1 \leq p^d + 1$ for some divisor $r$ of $d$, and each nontrivial complete quotient graph has order $p^d$. That is, $\Gamma$ satisfies \cite[Theorems 1.2 and 1.3]{AGP-quotcomp} and thus is one of the graphs in \cite[Tables 1 and 2]{AGP-quotcomp}.

We now prove Theorems \ref{theorem:mainthm1} and \ref{theorem:mainthm2}.

\begin{proof}[Proof of Theorem \ref{theorem:mainthm1}]
Suppose that $p \equiv 2 \mymod{3}$, $d$ is odd, and $d \not\equiv 0 \mymod{3}$. It follows from Theorem \ref{theorem:maintheorem} that $\Gamma$ is $G$-vertex-quasiprimitive. Conversely, if $p = 3$ or $p \equiv 1 \mymod{3}$ then $\Gamma$ is $G$-quotient-complete for $G_\mathbf{0}$ as in line 1 of Table \ref{table:G_0}. Suppose that $p \equiv 2 \mymod{3}$. If $d$ is even take $G_\mathbf{0}$ to be as in line 1 of Table \ref{table:G_0} with $\ell = d$ (so that $K \leq \GL(U)$), and if $d \equiv 0 \mymod{3}$ take $G_\mathbf{0}$ to be as in line 2 with $i = 1$ and $\ell = d$. In both cases $\Gamma$ is $G$-quotient-complete. Hence Theorem \ref{theorem:mainthm1} holds.
\end{proof}

\begin{proof}[Proof of Theorem \ref{theorem:mainthm2}]
This follows immediately from Theorem \ref{theorem:maintheorem}.
\end{proof}

\begin{example} \label{example:n=1}
Assume Hypothesis \ref{hypothesis} with $n = 1$ and $\pi(G_\mathbf{0}) = \GL_1(q)$, and let $\Gamma = \lsg(H)$. Identify $\GL_1(q)$ with $\F^\#_q$, and let $\F^\square_q = \big\{ c^2 \ \big| \ c \in \F^\#_q \big\}$. Then $\Gamma$ is $G$-arc-transitive only if $G_\mathbf{0}$ is one of $\big\langle \F^\#_q, x \big\rangle$, $\big\langle \F^\#_q, x, y \big\rangle$, and $\big\langle \F^\square_q, x, hy \big\rangle$ for some $h \in \F^\#_q \setminus \F^\square_q$. (Indeed, if $G_\mathbf{0}$ is as in line 2 or 5 of Table \ref{table:G_0} then $|\pi(G_\mathbf{0})| = (q-1)/3$, so $\pi(G_\mathbf{0})$ is intransitive on $H^\#$.) Applying Theorem \ref{theorem:maintheorem}, and noting that $d = \ell$ for all possible $G_\mathbf{0}$, we get the following:
	\begin{enumerate}[1.]
	\item The graph $\Gamma$ is $G$-arc-transitive and $G$-quotient-complete with exactly one nontrivial complete $G$-normal quotient, corresponding to the normal subgroup $T_{V_1}$, if and only if $p = 3$ and $G_\mathbf{0}$ is one of the groups $\big\langle \F^\#_q, x \big\rangle$, $\big\langle \F^\#_q, x, y \big\rangle$, and $\big\langle \F^\square_q, x, hy \big\rangle$ with $h \in \F^\#_q \setminus \F^\square_q$.
	\item The graph $\Gamma$ is $G$-arc-transitive and $G$-quotient-complete with exactly two nontrivial complete $G$-normal quotients, corresponding to the minimal normal subgroups $T_{V_c}$ and $T_{V_{c^{-1}}}$ with $|c| = 3$, if and only if one of the following holds:
		\begin{enumerate}[a.]
		\item $G_\mathbf{0} = \big\langle \F^\#_q, x \big\rangle$ and either $p \equiv 1 \mymod{3}$, or $p \equiv 2 \mymod{3}$ with $d$ even; or
		\item $G_\mathbf{0} = \big\langle \F^\square_q, x, hy \big\rangle$, $h \in \F^\#_q \setminus \F^\square_q$, $p \equiv 2 \mymod{3}$, $j$ odd, and $d$ even.
		\end{enumerate}
	\item For all other cases, $\Gamma$ is $G$-vertex-quasiprimitive.
	\end{enumerate}
The graphs and groups $G_\mathbf{0}$ in item 2.a. are precisely those in Example \ref{example:lsg}.
\end{example}


\section*{Acknowledgements}

The author thanks Prof. Pablo Spiga for introducing the graphs and corresponding automorphism groups in Example \ref{example:Cay} and Prof. John Bamberg for pointing out the connection to finite nets.

\end{document}